\documentclass[12pt,reqno]{article}

\usepackage[usenames]{color}
\usepackage{amssymb}
\usepackage{graphicx}
\usepackage{amscd}

\usepackage[shortlabels]{enumitem}

\usepackage[colorlinks=true,
linkcolor=webgreen,
filecolor=webbrown,
citecolor=webgreen]{hyperref}

\definecolor{webgreen}{rgb}{0,.5,0}
\definecolor{webbrown}{rgb}{.6,0,0}

\usepackage{color}
\usepackage{fullpage}
\usepackage{float}

\usepackage{graphics,amsmath,amssymb}
\usepackage{amsthm}
\usepackage{amsfonts}
\usepackage{latexsym}
\usepackage{epsf}

\def\suchthat{\, : \, }

\begin{document}

\theoremstyle{plain}
\newtheorem{theorem}{Theorem}
\newtheorem{corollary}[theorem]{Corollary}
\newtheorem{lemma}[theorem]{Lemma}
\newtheorem{proposition}[theorem]{Proposition}
\newtheorem{claim}[theorem]{Claim}

\theoremstyle{definition}
\newtheorem{definition}[theorem]{Definition}
\newtheorem{example}[theorem]{Example}
\newtheorem{conjecture}[theorem]{Conjecture}

\theoremstyle{remark}
\newtheorem{remark}[theorem]{Remark}

\author{Carlo Sanna\thanks{C.~Sanna is a member GNSAGA of the INdAM and of CrypTO, the Group of Cryptography and Number Theory of Politecnico di Torino.}\\
Department of Mathematical Sciences\\
Politecnico di Torino\\ 
Corso Duca degli Abruzzi 24\\
10129 Torino \\
Italy\\
\href{mailto:carlo.sanna.dev@gmail.com}{\tt carlo.sanna.dev@gmail.com}
\and
Jeffrey Shallit and Shun Zhang \\
School of Computer Science\\
University of Waterloo\\
Waterloo, ON  N2L 3G1 \\
Canada\\
\href{mailto:shallit@uwaterloo.ca}{\tt shallit@uwaterloo.ca}\\
\href{mailto:s385zhang@uwaterloo.ca}{\tt s385zhang@uwaterloo.ca}
}
\title{The Largest Entry in the Inverse of a\\ Vandermonde Matrix}

\maketitle

\begin{abstract}
We investigate the size of the largest entry (in absolute value) in the inverse of certain Vandermonde matrices.
More precisely, for every real $b > 1$, let $M_b(n)$ be the maximum of the absolute values of the entries of the inverse of the $n \times n$ matrix $[b^{i  j}]_{0 \leq i, j < n}$.
We prove that $\lim_{n \to +\infty} M_b(n)$ exists, and we provide some formulas for it.
\end{abstract}

\section{Introduction}

Let $\mathbf{a} = (a_0, a_1, \ldots, a_{n-1})$ be a list of $n$ real numbers.
The classical \emph{Vandermonde matrix} $V(\mathbf{a})$ is defined as follows:
\begin{equation*}
V(\mathbf{a}) := \left[
\begin{array}{ccccc}
1 & a_0 & a_0^2 & \cdots & a_0^{n-1} \\
1 & a_1 & a_1^2 & \cdots & a_1^{n-1} \\
\vdots & \vdots & \vdots & \ddots & \vdots \\
1 & a_{n-1} & a_{n-1}^2 & \cdots & a_{n-1}^{n-1}
\end{array} \right] .
\end{equation*}
As is well-known, the Vandermonde matrix $V(\mathbf{a})$ is invertible if and only if the $a_i$ are pairwise distinct.
See, for example,
\cite{Klinger:1967}.   

In what follows, $n$ is a positive integer and $b > 1$ is a fixed
real number.
Let us define the entries $c_{i,j,n}$ by
$$[c_{i,j,n}]_{0 \leq i, j < n} = 
V(b^0, b^1, b^2, \ldots, b^{n-1})^{-1},$$
and let $M_b(n) = \max_{0 \leq i, j < n} |c_{i,j,n}|$,
the maximum of the absolute values of the entries of $V(1, b, b^2, \ldots, b^{n-1})^{-1}$.    
The size of the entries of inverses of Vandermonde matrices have been studied for a long time (e.g., \cite{Gautschi:1962}).   
Recently, in a paper by the first two authors and Daniel~Kane~\cite{Kane&Sanna&Shallit:2019}, we needed to estimate $M_2 (n)$, and we proved 
that $M_2 (n) \leq 34$.
In fact, even more is true:  the limit
$\lim_{n \rightarrow \infty} M_2 (n)$ exists
and equals $3 \prod_{i \geq 2} \left( 1 + \frac{1}{2^i - 1} \right) \doteq 5.19411992918 \cdots$.
In this paper, we generalize this result, replacing $2$ with any real number greater than $1$.

Our main results are as follows:
\begin{theorem}
Let $b > 1$ and
$n_0 = \lceil\log_b(1+\frac{1}{b})\rceil$. 
Then $|c_{i,j,n}| \leq |c_{n_0, n_0,n}|$
for $i, j \geq n_0$.  Hence
$M_b (n) \in \{ |c_{i,j,n}| \suchthat 0 \leq i, j \leq n_0 \}$.
\label{thm1}
\end{theorem}

\begin{theorem}
Let $b\geq \tau = (1+\sqrt{5})/2$ and $n \geq 2$.   Then $M_b(n) \in \{ |c_{0,0,n}|,
|c_{1,1,n}| \} $.
\label{thm2}
\end{theorem}

\begin{theorem}
For all real $b > 1$ the limit $\lim_{n \rightarrow \infty} M_b(n)$ exists.
\label{thm:existence}
\end{theorem}


\section{Preliminaries}

For every real number $x$, and for all integers $0 \leq i, j < n$, let us define the power sum
\begin{equation*}
\sigma_{i,j,n}(x) := \sum_{\substack{0 \leq h_1 < \cdots < h_i < n \\ h_1, \dots, h_i \neq j}} x^{h_1 + \cdots + h_i} .
\end{equation*}
The following lemma will be useful in later arguments.

\begin{lemma}
Let $i,j,n$ be integers with $0 \leq i < n$, $0 \leq j < n - 1$, and let $x$ be a positive real number.
\begin{enumerate}[(a)]
\item 
If $x > 1$, then $\sigma_{i, j,n}(x) \geq \sigma_{i,j,n+1}(x)$.
\item
If $x < 1$, then $\sigma_{i, j,n}(x) \leq \sigma_{i,j+1,n}(x)$.
\end{enumerate}
\label{lem:sigma}
\end{lemma}

\begin{proof}
We have
\begin{equation*}
\sigma_{i, j+1,n}(x) - \sigma_{i,j,n}(x) = \sum_{(h_1, \ldots, h_i) \in S_{i,j,n}} x^{h_1 + \cdots + h_i} - \sum_{(h_1, \ldots, h_i) \in T_{i,j,n}} x^{h_1 + \cdots + h_i} ,
\end{equation*}
where
\begin{equation*}
S_{i,j,n} := \{0 \leq h_1 < \cdots < h_i < n \suchthat j \in \{h_1, \ldots, h_i\},\, j + 1 \notin \{h_1, \ldots, h_i\}\}
\end{equation*}
and
\begin{equation*}
T_{i,j,n} := \{0 \leq h_1 < \cdots < h_i < n \suchthat j \notin \{h_1, \ldots, h_i\},\, j + 1 \in \{h_1, \ldots, h_i\}\} .
\end{equation*}
Now there is a bijection $S_{i,j,n} \to T_{i,j,n}$ given by
\begin{equation*}
(h_1, \ldots, h_i) \mapsto (h_1, \ldots, h_{i_0-1}, h_{i_0} + 1, h_{i_0 + 1}, \ldots, h_i) ,
\end{equation*}
where $i_0$ is the unique integer such that $h_{i_0} = j$.
Hence, it follows easily that $\sigma_{i, j,n}(x) \geq \sigma_{i,j+1,n}(x)$ for $x > 1$, and $\sigma_{i, j,n}(x) \leq \sigma_{i,j+1,n}(x)$ for $x < 1$.
\end{proof}

Recall the following formula for the entries of the inverse of a Vandermonde matrix (see, e.g., \cite[\S 1.2.3, Exercise~40]{Knuth:1997}).

\begin{lemma}\label{thm:lagrange}
Let $a_0, \dots, a_{n-1}$ be pairwise distinct real numbers.
If $V(a_0, a_1, \dots, a_{n-1}) = [c_{i,j}]_{0 \leq i , j < n}$ then
\begin{equation*}
c_{n-1,j} X^{n-1} + c_{n-2,j} X^{n-2} +
\cdots + c_{0,j} X^0 = \prod_{\substack{0 \leq i < n \\ i \neq j}} \frac{X - a_i}{a_j - a_i} .
\end{equation*}
\end{lemma}

For $0 \leq i, j < n$ define
\begin{equation*}
\pi_{j,n} := \prod_{\substack{0 \leq h < n \\ h \neq j}} |b^j - b^h| .
\end{equation*}

We now obtain a relationship between the entries of
$V(b^0, b^1,\ldots, b^{n-1})^{-1}$ and $\sigma_{i,j,n}$ and $\pi_{j,n}$.
\begin{lemma}
Let $V(b^0, b^1, \ldots, b^{n-1})^{-1} =
[c_{i,j,n}]_{0 \leq i, j < n} $.  Then
\begin{equation}
|c_{i,j,n}| = 
\frac{\sigma_{n-i-1,j,n}}{\pi_{j,n}}
\label{fund}
\end{equation}
for $0 \leq i,j < n$.
\end{lemma}

\begin{proof}
By Lemma~\ref{thm:lagrange}, we have
\begin{equation*}
\prod_{\substack{0 \leq h < n \\ h \neq j}} \frac{X - b^h}{b^j - b^h}  =
\sum_{0 \leq i < n} c_{i,j,n} X^i .
\end{equation*}
which in turn, by Vieta's formulas, gives
\begin{equation}
c_{n-i-1,j,n} = (-1)^i \left( \prod_{\substack{0 \leq h < n \\ h \neq j}} \frac1{b^j - b^h} \right) \sum_{\substack{0 \leq h_1 < \cdots < h_i < n \\ h_1, \ldots, h_i \neq j}} b^{h_1 + \cdots + h_i}
\label{vieta}
\end{equation}
for $0 \leq i < n$.   The result now follows by the definitions
of $\sigma$ and $\pi$.
\end{proof}

Next, we obtain some inequalities for $\pi$.
\begin{lemma}
Define $n_0 = \lceil\log_b(1+\frac{1}{b})\rceil$.  Then
$$ \pi_{j,n} \leq \pi_{j+1,n} \quad \text{for $n_0 \leq j < n$} .$$
\label{pineq}
\end{lemma}

\begin{proof}
For $0 \leq j < n - 1$, we have
\begin{align*}
\pi_{j+1,n} := \prod_{\substack{0 \leq h < n \\ h \neq j + 1}} |b^{j+1} - b^h| = b^{n-1}\prod_{\substack{0 \leq h < n \\ h - 1 \neq j}} |b^{j} - b^{h-1}| = \frac{b^{n+j-1} - b^{n-2}}{b^{n-1} - b^j} \pi_{j,n} .
\end{align*}
A quick computation shows that the inequality
\begin{equation*}
\frac{b^{n+j-1} - b^{n-2}}{b^{n-1} - b^j} \geq 1
\end{equation*}
is equivalent to
\begin{equation*}
b^j \geq \frac{b^{n-1} + b^{n-2}}{b^{n-1} + 1} .
\end{equation*}
Let $n_0$ be the minimum positive integer such that
$b^{n_0} \geq 1 + \frac1{b}$.  Then $ n_0 = \lceil\log_b(1+\frac{1}{b})\rceil$.
Hence, for $n_0 \leq j < n$, we have
\begin{equation*}
b^j \geq 1 + \frac1{b} > \frac{b^{n-1} + b^{n-2}}{b^{n-1} + 1} ,
\end{equation*}
so that 
\begin{equation}
\pi_{j,n} \leq \pi_{j+1,n} \quad \text{for $n_0 \leq j < n$}.
\label{pi-ineq}
\end{equation}
\end{proof}

Finally, we have the easy
\begin{lemma}
For $0 \leq i , j < n$ we have 
$c_{i,j,n} = c_{j,i,n}$.
\label{symm}
\end{lemma}

\begin{proof}
$V(b^0, b^1,\ldots, b^{n-1})$ is a symmetric matrix, so its
inverse is also.
\end{proof}

\section{Proof of Theorem~\ref{thm1}}

\begin{proof}
Suppose $i, j \geq n_0$.  Then
\begin{align*}
|c_{i,j,n}| &= \frac{\sigma_{n-i-1,j,n}}{\pi_{j,n}}  \quad
	\text{(by \eqref{fund})} \\
& \leq \frac{\sigma_{n-i-1,n_0,n}}{\pi_{j,n}}  \quad
	\text{(by Lemma~\ref{lem:sigma} (a))} \\
& \leq \frac{\sigma_{n-i-1,n_0,n}}{\pi_{n_0,n}}  \quad
        \text{(by Lemma~\ref{pineq})} \\
& = |c_{i,n_0,n}| \quad \text{(by \eqref{fund})} ,
\end{align*}
and so we get
\begin{equation}
|c_{i,j,n}| \leq |c_{i,n_0,n}| .
\label{ine}
\end{equation}
But 
\begin{equation}
c_{i,n_0, n} = c_{n_0, i, n}
\label{ine2}
\end{equation}
by Lemma~\ref{symm}.
Make the substitutions
$n_0$ for $i$ and $i$ for $j$ in \eqref{ine} to get
\begin{equation}
|c_{n_0,i,n}| \leq |c_{n_0, n_0, n}|.
\label{ine3}
\end{equation}
The result now follows by combining Eqs.~\eqref{ine}, \eqref{ine2}, and
\eqref{ine3}.
\end{proof}

\section{Proof of Theorem~\ref{thm2}}

\begin{proof}
Since $b \geq \tau$, it follows that $b \geq 1 + 1/b$.  
Hence in Theorem~\ref{thm1} we can take $n_0 = 1$, and this
gives $M_b (n) \in
\{ |c_{0,0,n}|, |c_{1,0,n}|, |c_{0,1,n}|, |c_{1,1,n}| \}$.
However, by explicit calculation, we have
\begin{align*}
\sigma_{n-1,1,n} &= b^{ n(n-1)/2 - 1} \\
\sigma_{n-2,1,n} &= b^{ n(n-1)/2 - 1} + \sum_{(n-1)(n-2)/2 - 1 \leq
i \leq n(n-1)/2 - 3 } b^i,
\end{align*}
so that
\begin{equation}
\sigma_{n-1,1,n} \leq \sigma_{n-2,1,n}.
\label{n12}
\end{equation}
Hence
\begin{align*}
|c_{1,0,n}| &= |c_{0,1,n}| \quad \text{(by Lemma~\ref{symm})} \\
&= \frac{\sigma_{n-1,1,n}}{\pi_{1,n}} \quad
	\text{(by \eqref{fund})} \\
&\leq \frac{\sigma_{n-2,1,n}}{\pi_{1,n}}  \quad \text{(by \eqref{n12})} \\
&= |c_{1,1,n}| \quad \text{(by \eqref{fund})},
\end{align*}
and the result follows.
\end{proof}

\section{Proof of Theorem~\ref{thm:existence}}

\begin{proof}
We have
\begin{align*}
|c_{i,j,n}| &= \frac{\sigma_{n-i-1,j,n}}{\pi_{j,n}} \\
&=
\frac{\sigma_{n-i-1,j,n} (b)}{\prod_{\substack{0 \leq h < n \\ h \neq j}} |b^j - b^h|} \\
&= \frac{\sigma_{n-i-1,j,n} (b)}{\prod_{\substack{0 \leq h < n \\ h \neq j}} \ (b^h \cdot |b^{j-h} - 1|)} \\
&= \frac{\sigma_{n-i-1,j,n} (b)}{b^{n(n-1)/2 - j}} \cdot \frac{1}{\prod_{\substack{0 \leq h < n \\ h \neq j}} |b^{j-h} - 1|} \\
&= \sigma_{i,j,n}(b^{-1}) \frac{1}{\prod_{\substack{0 \leq h < n \\ h \neq j}} |b^{j-h} - 1|},
\end{align*}
where the equality
$$\frac{\sigma_{n-i-1,j,n} (b)}{b^{n(n-1)/2 - j}}  = \sigma_{i,j,n}(b^{-1})$$
arises from the one-to-one correspondence between the subsets of $\{0,1,\ldots, n-1\} - \{j\}$
of cardinality $i$ and those of cardinality $n-1-i$.

For $x < 1$ define
$$ \sigma_{i,j,\infty} (x)  =
\sum_{\substack{0 \leq h_1 < \cdots < h_i < \infty
\\ h_1, \ldots, h_i \neq j}} \frac1{x^{h_1 + \cdots + h_i}}.$$

Hence the limits
\begin{align}
\ell_{i,j} & := \lim_{n \to +\infty} |c_{i,j,n}|  \nonumber \\
&= \lim_{n \to +\infty} \sigma_{i,j,n}(b^{-1}) \frac{1}{\prod_{\substack{0 \leq h < n \\ h \neq j}} |b^{j-h} - 1|} \nonumber \\
&= \sigma_{i,j,\infty} (b^{-1}) \left(\prod_{1 \leq s \leq j} \frac1{b^s - 1} \right) \left(\prod_{t \geq 1}  \frac1{1 - b^{-t}} \right) 
\label{limits}
\end{align}
exist and are finite.

From Theorem~\ref{thm1} we see that
\begin{equation*}
\lim_{n \to +\infty} M_b(n) = 
\max_{0 \leq i \leq j < n_0} \ 
\lim_{n \to +\infty} |c_{i,j,n}|
= \max_{0 \leq i \leq j \leq n_0}\ \ell_{i,j} ,
\end{equation*}
and the proof is complete.
\end{proof}

From this theorem we can explicitly compute 
$\lim_{n \to +\infty} M_b(n)$ for $b \geq \tau$.

\begin{corollary}
Let $\alpha \dot= 2.324717957$ be the real zero of
the polynomial $X^3-3X^2+2X-1$.
\begin{enumerate}[(a)]   
\item If $b \geq \alpha$, then
 $\lim_{n \rightarrow \infty} M_b (n) = \prod_{t \geq 1} (1-b^{-t})^{-1}$.
 
 \item If $\tau \leq b \leq \alpha$, then
 $\lim_{n \rightarrow \infty} M_b (n) = 
 \frac{b^2-b+1}{b(b-1)^2} \prod_{t \geq 1} (1-b^{-t})^{-1}$.
 \end{enumerate}
\end{corollary}

\begin{proof}
From Theorem~\ref{thm2} we know that for
$b \geq \tau$ we have 
$\lim_{n \rightarrow \infty} M_b (n) \in \{ \ell_{0,0}, \ell_{1,1} \} $.   Now an easy
calculation based on \eqref{limits} shows that
\begin{align*}
\ell_{0,0} &= \prod_{t \geq 1} (1-b^{-t})^{-1} \\
\ell_{1,1} &= \frac{b^2-b+1}{b(b-1)^2} \  \prod_{t \geq 1} (1-b^{-t})^{-1} .
\end{align*}
By solving the equation
$\frac{b^2-b+1}{b(b-1)^2} = 1$,
we see that for $b \geq \alpha$ we have
$\ell_{0,0} \geq \ell_{1,1}$, while
if $\tau \leq b \leq \alpha$ we have
$\ell_{1,1} \geq \ell_{0,0}$.  This proves
both parts of the claim.
\end{proof}

\begin{remark}
The quantity $M_b (n)$ converges rather slowly to its limit when $b$ is close to $1$.   The
following table gives some numerical estimates for $M_b (n)$.
\begin{table}[H]
\begin{center}
\begin{tabular}{c|l}
$b$ & $\lim_{n \rightarrow \infty} M_b (n)$ \\
\hline
3 & 1.785312341998534190367486\\
$\alpha \doteq 2.3247$ & 2.4862447382651613433\\
2 & 5.194119929182595417 \\
$\tau \doteq 1.61803$ & 26.788216012030303413\\
1.5 & 67.3672156 \\
1.4 & 282.398 \\
1.3 & 3069.44 \\
1.2 & 422349.8 
\end{tabular}
\end{center}
\end{table}
\end{remark}

\section{Final remarks}

We close with a conjecture we have been unable to prove.

\begin{conjecture}
Let $b > 1$ and
$n_0 = \lceil\log_b(1+\frac{1}{b})\rceil$. 
Then, for all sufficiently large $n$, we have
$M_b (n) = |c_{i,i,n}|$ for some $i$, $0 \leq i \leq n_0$.
\label{thm:diagonal}
\end{conjecture}

\end{document}